\documentclass[preprint,12pt]{elsarticle}
\usepackage{graphics}
\usepackage{graphicx}
\usepackage{epsfig}
\usepackage{amssymb}
\usepackage{amsthm}
\usepackage[displaymath, mathlines,pagewise]{lineno}
\usepackage{latexsym}
\usepackage{multirow}
\usepackage{mathtools}
\usepackage[dvips]{color}
\usepackage{psfrag}
\usepackage{slashbox}
\usepackage{multirow}
\usepackage{caption}
\usepackage{subcaption}
\usepackage{epsf,amsfonts}
\usepackage{float}
\numberwithin{equation}{section}
\newtheorem{theorem}{Theorem}[section]

\newtheorem{example}[theorem]{Example}

\journal{?}

\begin{document}

\begin{frontmatter}

\title{An inverse problem in estimating the time
dependent source term and initial temperature simultaneously
 by  the polynomial regression and conjugate gradient method}

\author{Arzu Erdem Co\c{s}kun}
\address{erdem.arzu@gmail.com}
\address{Kocaeli University, Faculty of Arts and Sciences, Department of Mathematics}
\address{Umuttepe Campus, 41380, Kocaeli - TURKEY}

\begin{abstract}
From the final and interior temperature measurements identifying the source term with initial temperature simultaneously is an inverse heat conduction problem which is a kind of ill-posed. The optimal control framework has been found to be effective in dealing with these problems. However, they require to find the gradient information. This idea has been employed in this research. We derive the gradient of Tikhonov functional and establish the stability of the minimizer from the necessary condition. The stability and effectiveness of evolutionary algorithm are presented for various test examples.
\end{abstract}

\begin{keyword}
Inverse source problem, parabolic problem, Tikhonov functional, polynomial regression, conjugate gradient method

\end{keyword}

\end{frontmatter}


\section{Introduction}
In recent years, several researchers have reported for the solvability of the parabolic inverse problems of finding a solution with an unknown right-hand side. For instance, the inverse source problem when the final and (or) integral overdetermination condition are given have been studied in \cite{BadiaDuongHamdi,JohanssonLesnic,KozhanovSafiullova,YanFuDou,YanYangFu,YangDengYuLuo}. Borukhov and Vabishchevich \cite{BorukhovVabishchevich} have reconstructed the right-hand side of a parabolic equation using a solution specified at internal points. Abasheeva \cite{Abasheeva} has proved the existence and uniqueness of the solution to the inverse source problem. Ashyralyev, Erdogan and Demirdag \cite{AshyralyevErdoganDemirdag} have presented a stable difference schemes of first and second orders of accuracy for the inverse source problem. A numerical method by employing a new idea of fictitious time integration method for backward heat conduction problem has been given by Chang and Liu \cite{ChangLiu}. The unknown initial condition for a parabolic  system has been concerned by Tadi \cite{Tadi}.  Masooda K. and Yousufa \cite{MasoodaYousufa} have proposed a class of numerical schemes based on positivity-preserving Pad\'{e} approximations to solve initial inverse problems in the heat equation. The initial temperature and a boundary coefficient have been simultaneously determined from the final overdetermination and a priori known in a small sub-domain in \cite{ChoulliYamamoto}. Erdem \cite{Erdem}  study two inverse problems relating to reconstruction of the diffusion coefficient $k(x)$, appearing in a linear partial parabolic equation $u_t=(k(x)u_x)_x$ concerned through overposed data $ u(x,T)$ and non-local boundary condition $\int_0^T u(x,t)dt$. An inverse source problem of identification of $F(t)$ function in the linear parabolic equation $u_t= u_{xx} + F(t)$ and $u_0(x)$ function as the initial condition from the measured final data and local boundary data based on the optimal control framework by Green's function has been presented in \cite{Azizbayov,Erdem1,Polidoro,XuYang}. Source terms $w:=\{F(x,t),p(t)\}$ in the linear parabolic equation $u_t= u_{xx} + F(x,t)$ and Robin boundary condition $-u_x(l,t)= \nu [u(l,t) - p(t)]$ from the  measured final data and the measurement of the temperature in a subregion has been investigated in \cite{Erdem2} based on the minimization of Tikhonov functional.

We consider one-dimensional heat conduction specified by the differential equation:
\begin{equation}
\frac{\partial u}{\partial t}=\frac{\partial ^{2}u}{\partial x^{2}}%
+F(t),~~(x,t)\in (0,L)\times (0,t_{f}),  \label{1_e}
\end{equation}
where $u(x,t)$ describes  the temperature field depending on the spacewise variable $x$ and the time $t$, $F(t)$ is the internal heat source, $L$ and $t_{f}$ are given positive constants . The initial temperature are set as
\begin{equation}
u(x,0)=u_{0}(x),~~x\in (0,L),  \label{1_i}
\end{equation}%
and two boundary conditions are defined:
\begin{equation}
u(0,t)=u(L,t)=0,~~t\in (0,t_{f}].  \label{1_b}
\end{equation}

When the internal heat source $F(t)$ and initial temperature $u_0(x)$ are given, the problem (\ref{1_e})-(\ref{1_b}) is referred as  the direct problem. The problem of identifying the unknown internal heat source $F(t)$ and initial temperature $u_0(x)$ through the following additional information is considered as the inverse problem:
\begin{eqnarray}
u_{f}(x) &=&u(x,t_{f}),~~x\in (0,L),  \label{2} \\
u^{\ast }(t) &=&u(x^{\ast },t),~~t\in (0,T],  \label{2_2}
\end{eqnarray}
where $x^*$ is an interior point.

By the method of separation of variables, the solution of (\ref{1_e})-(\ref{1_b})
can be written as
\begin{eqnarray*}
u(x,t)=\int_{0}^{L}
u_0(\xi)G(x,\xi,t)d\xi+\int_0^tF(\tau)H(x,t-\tau) d\tau,\label{4g}
\end{eqnarray*}
where the form of the Green's function is given by
\begin{eqnarray*}
G(x,\xi,t)=\frac{2}{L}\sum_{n=1}^\infty \sin(\lambda_n
x) \sin(\lambda_n \xi)\exp(-\lambda_n^2 t),\label{5g}
\end{eqnarray*}
with $\lambda_n=\frac{n\pi}{L}$ and the function $H(x,t)$ is
expressed in terms of Green's function as
\begin{eqnarray*}
H(x,t)=\int_{0}^{L}
G(x,\xi,t)d\xi=\frac{4}{L}\sum_{n=1}^\infty
\frac{1}{\lambda_{2n-1}}\sin(\lambda_{2n-1} x)
\exp(-\lambda_{2n-1}^2 t),\label{6g}
\end{eqnarray*}

For any given $u_0(x)\in H^{2+l}(0,L),~~f(t)\in
H^{l/2}(0,T),~\alpha \in (0,1)$ and the consistency condition
$u_0(0)=u_0(L)=0$ is satisfied, there exists a unique solution
$u(x,t)\in H^{2+l,1+l/2}([0,L]\times [0,T])$, to the
problem (\ref{1_e})-(\ref{1_b}), \cite{Ladyzhenskaya}.

In this paper,  the unknown internal heat source $F(t)$ and initial temperature $u_0(x)$ are represented a simplified version of the polynomial regression model:
\begin{eqnarray}
F(t) &=&\sum_{k=1}^{N_{t}}\phi _{k}t^{k-1},~0\leq t\leq t_{f},  \label{3} \\
u_{0}(x) &=&\sum_{m=1}^{N_{x}}\theta _{m}x^{m-1},~0\leq x\leq L \label{3_2}
\end{eqnarray}
We aim to find the parameters $\boldsymbol{\phi}=[\phi_1,\phi_2,...,\phi_{N_t}]^T $ and $\boldsymbol{\theta}=[\theta_1,\theta_2,...,\theta_{N_x}]^T $ which response best matches the recorded data. The solution of present inverse problem is to be sought in such a way that the Tikhonov objective function $S_\alpha(\boldsymbol{\phi},\boldsymbol{\theta})$ is minimized:
\begin{eqnarray}
S_{\alpha }(\boldsymbol{\phi },\boldsymbol{\theta }) &=&\sum_{i=1}^{I_{x}}%
\left[ u_{f}(x_{i})-u(x_{i},t_{f};\boldsymbol{\phi },\boldsymbol{\theta })%
\right] ^{2}+\sum_{j=1}^{I_{t}}\left[ u^{\ast }(t_{j})-u(x^{\ast },t_{j};%
\boldsymbol{\phi },\boldsymbol{\theta })\right] ^{2}  \nonumber \\
&&+\alpha \sum_{i=1}^{I_{x}}\left[ \sum_{m=1}^{N_{x}}\theta _{m}x_{i}^{m-1}%
\right] ^{2}+\alpha \sum_{j=1}^{I_{t}}\left[ \sum_{k=1}^{N_{t}}\phi _{k}t_{j}^{k-1}\right] ^{2}  \label{4}
\end{eqnarray}
where $\alpha>0$ is the regularization parameter.

Here, we present an efficient solution for the inverse problem (\ref{1_e})-(\ref{2_2}). It combines the iterative-type and Tikhonov regularization methods. The solution method aims to minimize the objective functional (\ref{4}). Further, the gradient of (\ref{4}) is defined, explicitly. Then the Conjugate Gradient method is proposed to solve resulting minimization problem.

The paper is organized as follows. In Section 2, the gradient  of the Tikhonov objective functional has been obtained.  The necessity condition has been presented for the stability  of solutions in the presence of measurement noise in Section 3. In Section 4, we design the algorithm and the obtained theoretical results will be tested practically in terms of numerical experiments.

\section{ Mathematical relations}
The analytical solution of (\ref{1_e})-(\ref{1_b}) is obtained as
\begin{eqnarray}
u(x,t) &=&\int_{0}^{L}u_{0}(\xi )G(x,\xi ,t)d\xi +\int_{0}^{t}F(\tau )H(x,t-\tau )d\tau   \label{5}
\end{eqnarray}
where Green's function with eigenvalue $\lambda _{n}=n\pi /L$ is given by%
\begin{equation}
G(x,\xi ,t)=\frac{2}{L}\sum_{n=1}^{\infty }\sin \left( \lambda _{n}x\right) \sin \left( \lambda _{n}\xi \right) \exp \left( -\lambda _{n}^{2}t\right) \label{6}
\end{equation}%
The functions $H$ is represented in terms of Green's
function given by%
\begin{gather}
H\left( x,t\right) =\int_{0}^{L}G(x,\xi ,t)d\xi =\frac{4}{L}%
\sum_{n=1}^{\infty }\frac{1}{\lambda _{2n-1}}\sin (\lambda _{2n-1}x)\exp (-\lambda _{2n-1}^{2}t),  \label{7}
\end{gather}%
By substituting (\ref{3}) and (\ref{3_2}) into the analytical solution at
the measurement points we have%
\begin{gather}
u(x,t_{f};\boldsymbol{\phi },\boldsymbol{\theta })=\sum_{m=1}^{N_{x}}\theta _{m}\int_{0}^{l}\xi ^{m-1}G(x,\xi ,t_{f})d\xi +\sum_{k=1}^{N_{t}}\phi
_{k}\int_{0}^{t_{f}}\tau ^{k-1}H(x,t_{f}-\tau )d\tau   ,  \label{10} \\
u(x^{\ast },t;\boldsymbol{\phi },\boldsymbol{\theta })=\sum_{m=1}^{N_{x}}%
\theta _{m}\int_{0}^{l}\xi ^{m-1}G(x^{\ast },\xi ,t)d\xi +\sum_{k=1}^{N_{t}}\phi _{k}\int_{0}^{t}\tau ^{k-1}H(x^{\ast },t-\tau )d\tau
 ,  \label{11}
\end{gather}%
and differentiating the above result with respect to $\phi _{k}$ and $\theta _{m}$ we obtain the following expression for the sensitivity coefficients for
the parameters $\phi _{k},k=1,2,...,N_{t}$ and $\theta _{m},m=1,2,...,N_{x}:$%
\begin{gather}
J_{m}^{1,1}\left( x\right) =\frac{\partial u(x,t_{f};\boldsymbol{\phi },%
\boldsymbol{\theta })}{\partial \theta _{m}}=\int_{0}^{l}\xi ^{m-1}G(x,\xi
,t_{f})d\xi ,  \label{12} \\
J_{m}^{2,1}\left( t\right) =\frac{\partial u(x^{\ast },t;\boldsymbol{\phi },%
\boldsymbol{\theta })}{\partial \theta _{m}}=\int_{0}^{l}\xi ^{m-1}G(x^{\ast
},\xi ,t)d\xi ,  \label{13} \\
J_{k}^{1,2}\left( x\right) =\frac{\partial u(x,t_{f};\boldsymbol{\phi },%
\boldsymbol{\theta })}{\partial \phi _{k}}=\int_{0}^{t_{f}}\tau
^{k-1}H(x,t_{f}-\tau )d\tau ,  \label{14} \\
J_{k}^{2,2}\left( t\right) =\frac{\partial u(x^{\ast },t;\boldsymbol{\phi },%
\boldsymbol{\theta })}{\partial \phi _{k}}=\int_{0}^{t}\tau ^{k-1}H(x^{\ast },t-\tau )d\tau .  \label{15}
\end{gather}%
\begin{figure}[H]
  \begin{subfigure}[b]{0.55\textwidth}
          \centering
          \includegraphics[width=\textwidth]{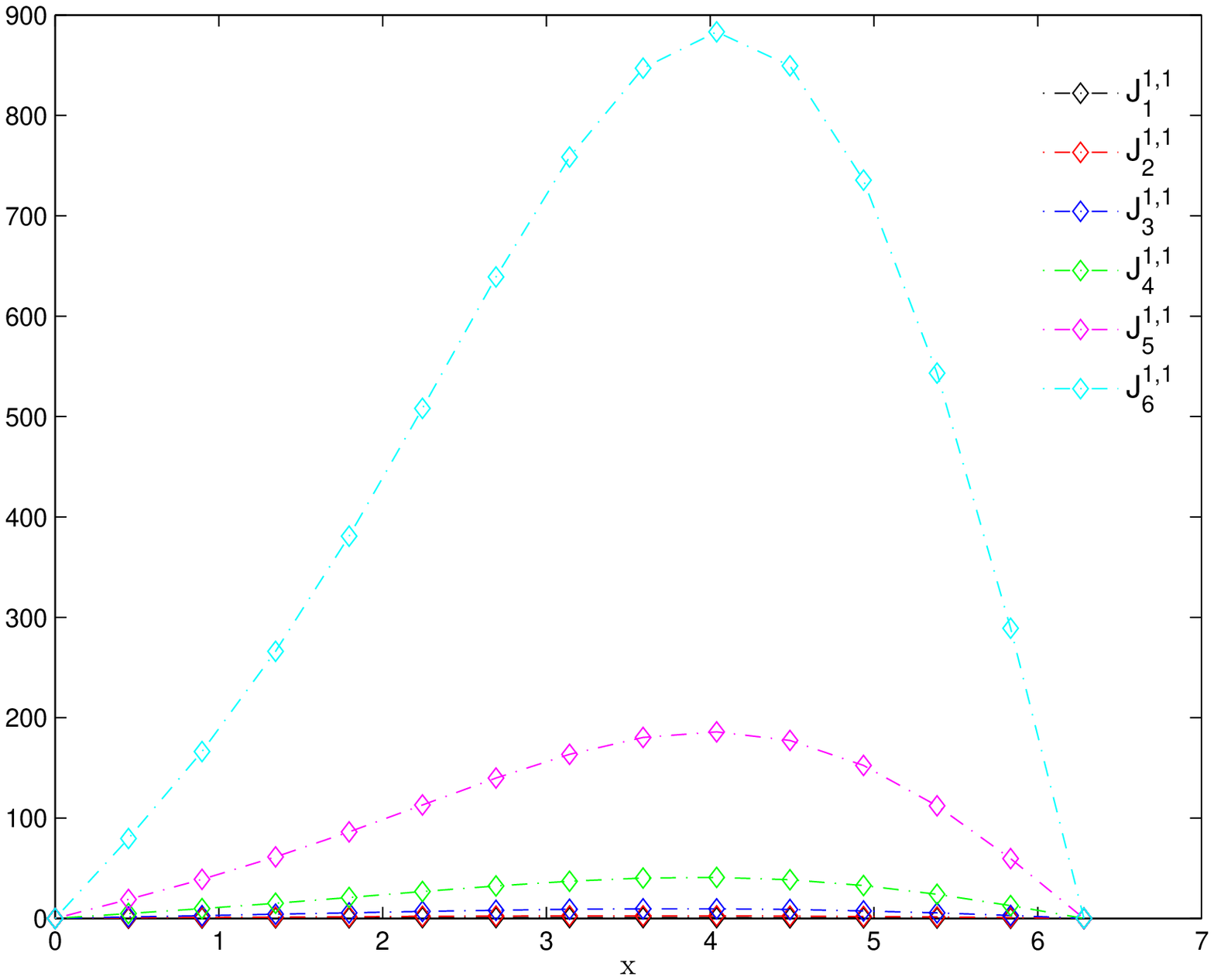}
          \caption{}{\label{J:fig1-a}}
  \end{subfigure}%
  \begin{subfigure}[b]{0.55\textwidth}
          \centering
          \includegraphics[width=\textwidth]{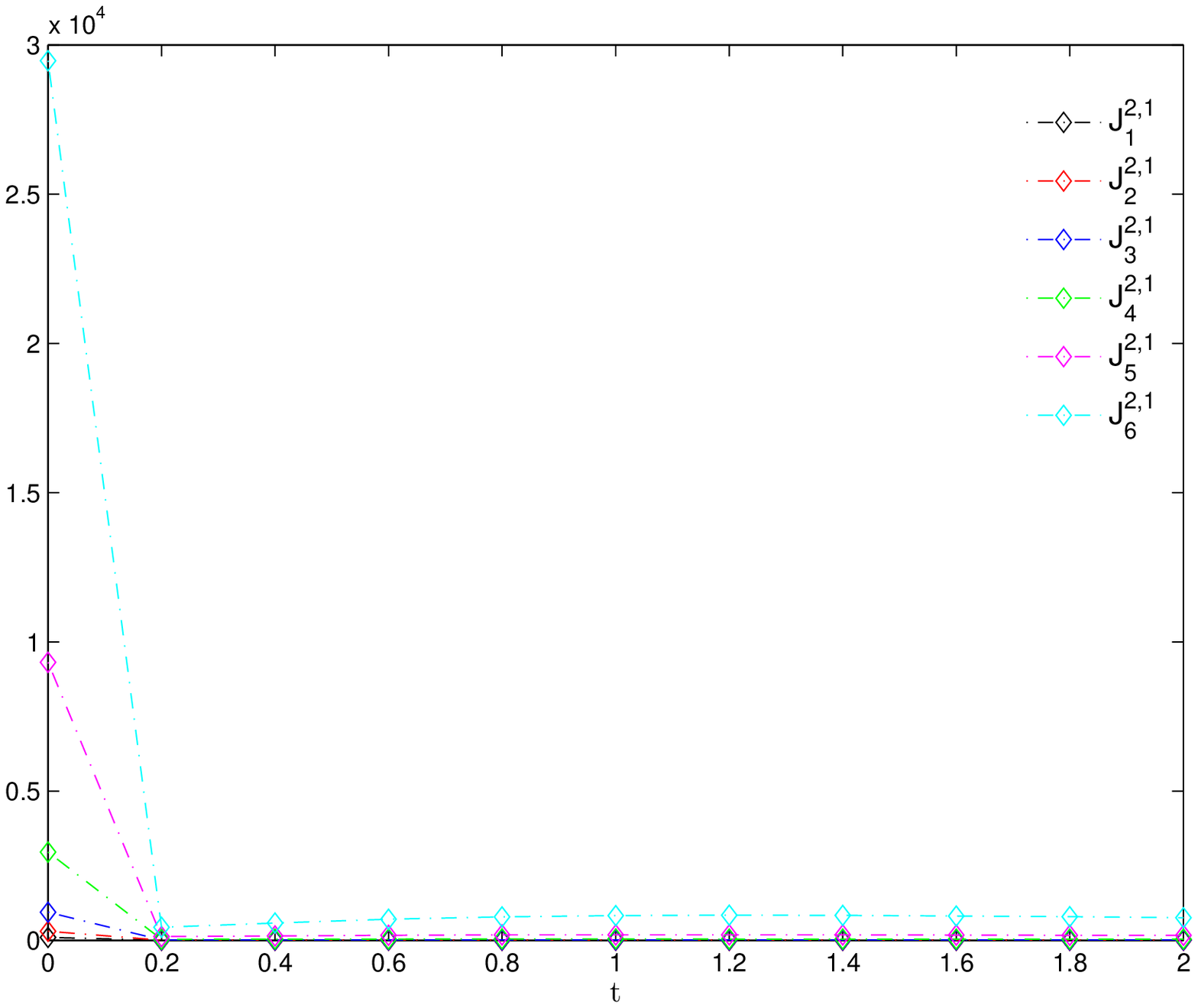}
          \caption{}{\label{J:fig1-b}}
  \end{subfigure}\\
  \begin{subfigure}[b]{0.55\textwidth}
          \centering
          \includegraphics[width=\textwidth]{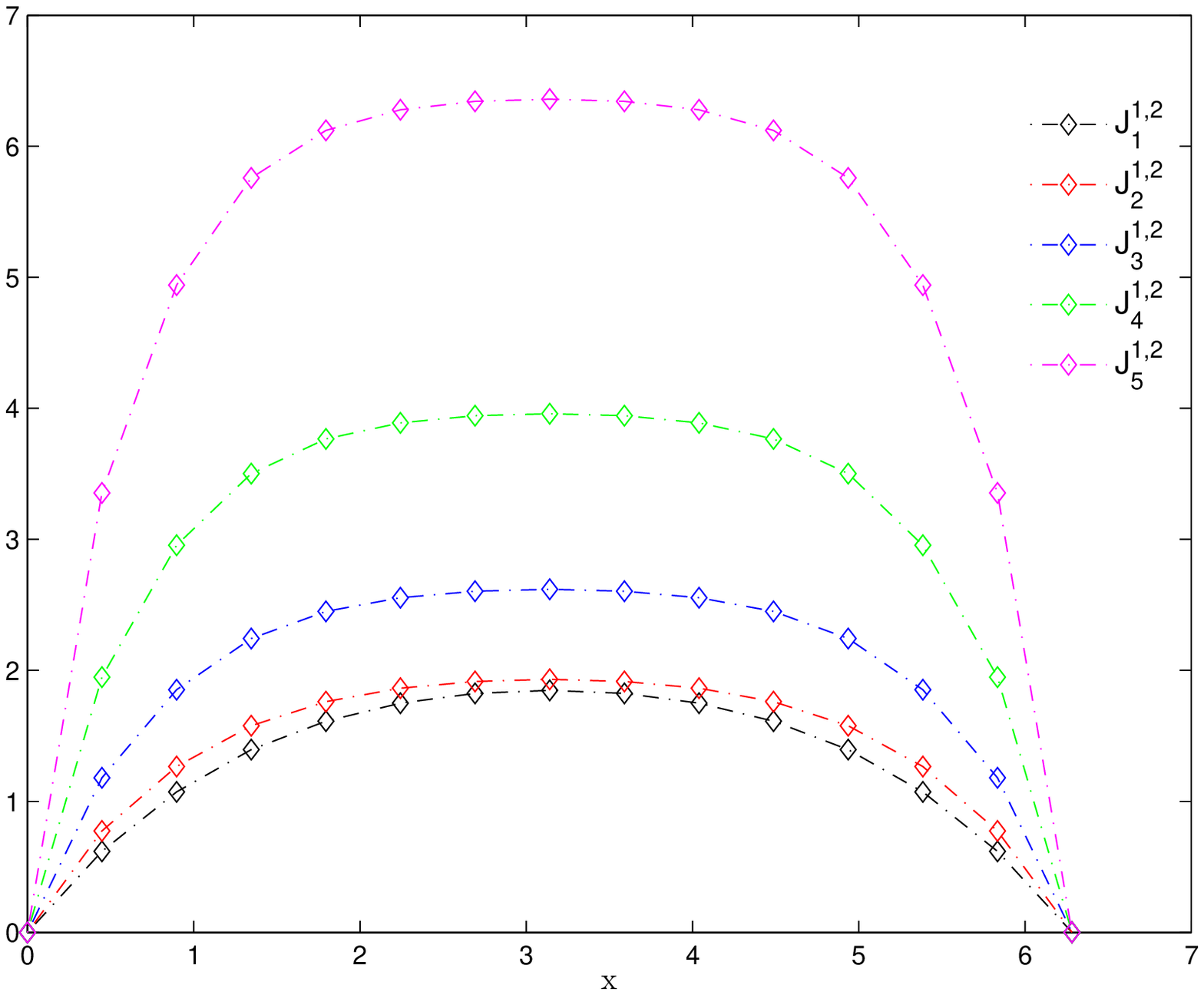}
          \caption{}{\label{J:fig1-c}}
  \end{subfigure}%
  \begin{subfigure}[b]{0.55\textwidth}
          \centering
          \includegraphics[width=\textwidth]{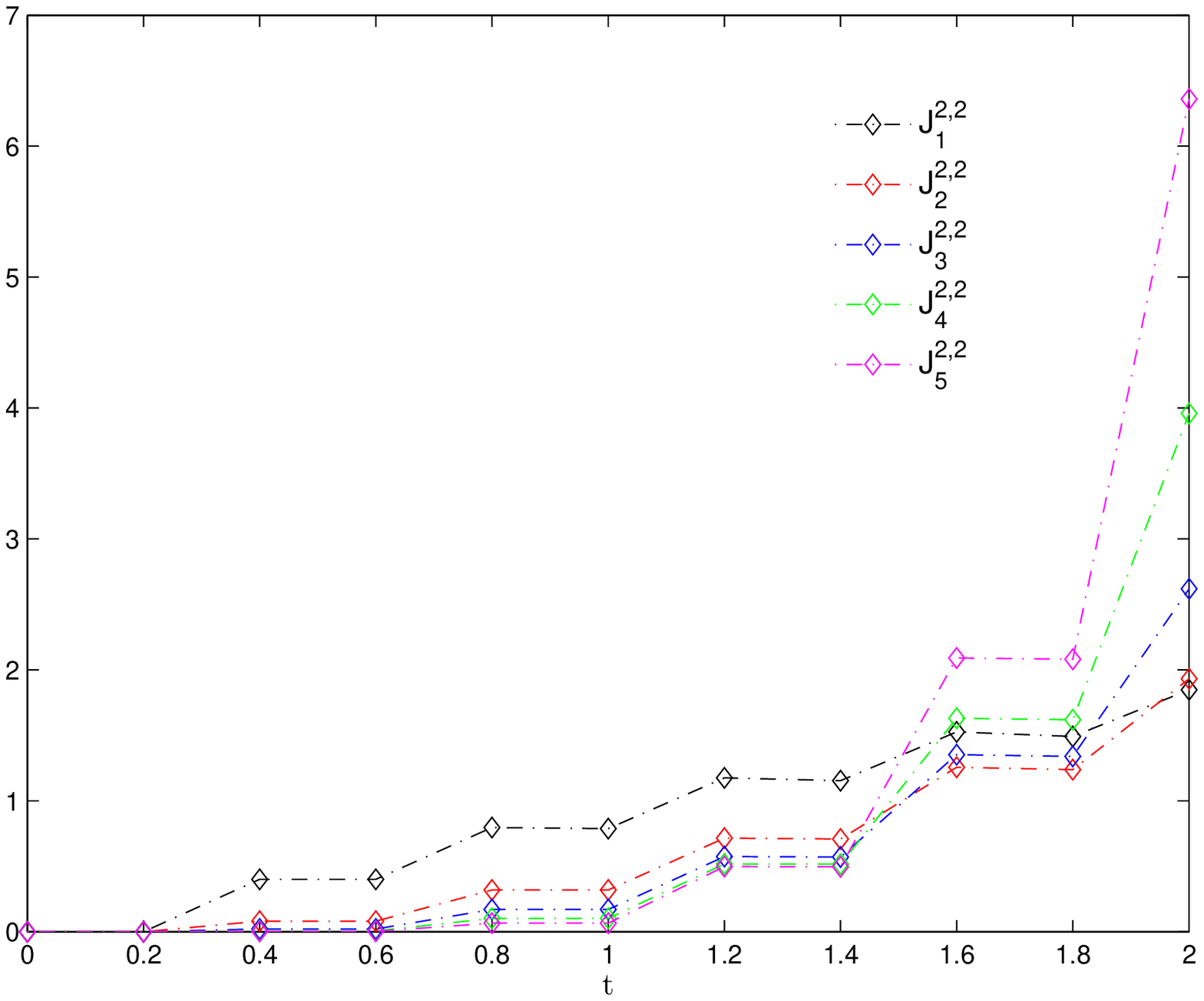}
          \caption{}{\label{J:fig1-d}}
  \end{subfigure}
  \caption{(\subref{J:fig1-a}) shows the sensitivity coefficient
  $J_{m}^{1,1}\left( x\right),~m=1,2,3,4,5,6$;\\
  (\subref{J:fig1-b}) shows the sensitivity coefficient
  $J_{m}^{2,1}\left( t\right),~m=1,2,3,4,5,6$;\\
  (\subref{J:fig1-c}) shows the sensitivity coefficient
  $J_{k}^{1,2}\left( x\right),~k=1,2,3,4,5$;\\
  (\subref{J:fig1-d}) shows the sensitivity coefficient $J_{k}^{2,2}\left( t\right),~k=1,2,3,4,5$ where $L=2\pi,~t_f=2,~x^*=\pi$}
  \label{J:fig1}
\end{figure}
By differentiating equation (\ref{4}) with respect to the unknown parameters $\theta _{m},m=1,2,...,N_{x}$ and $\phi _{k},k=1,2,...,N_{t}$ we obtain
\begin{gather}
\frac{\partial S_{\alpha }(\boldsymbol{\phi },\boldsymbol{\theta })}{%
\partial \phi _{k}}=-2\sum_{i=1}^{I_{x}}\left[ u_{f}(x_{i})-u(x_{i},t_{f};%
\boldsymbol{\phi },\boldsymbol{\theta })\right] J_{k}^{1,2}(x_{i})  \notag \\
-2\sum_{j=1}^{I_{t}}\left[ u^{\ast }(t_{j})-u(x^{\ast },t_{j};\boldsymbol{%
\phi },\boldsymbol{\theta })\right] J_{k}^{2,2}\left( t_{j}\right) +2\alpha
\sum_{j=1}^{I_{t}}t_{j}^{k-1}\left[ \sum_{k=1}^{N_{t}}\phi _{k}t_{j}^{k-1}%
\right]  \label{16} \\
\frac{\partial S_{\alpha }(\boldsymbol{\phi },\boldsymbol{\theta })}{%
\partial \theta _{m}}=-2\sum_{i=1}^{I_{x}}\left[ u_{f}(x_{i})-u(x_{i},t_{f};%
\boldsymbol{\phi },\boldsymbol{\theta })\right] J_{m}^{1,1}(x_{i})  \notag \\
-2\sum_{j=1}^{I_{t}}\left[ u^{\ast }(t_{j})-u(x^{\ast },t_{j};\boldsymbol{%
\phi },\boldsymbol{\theta })\right] J_{m}^{2,1}\left( t_{j}\right) +2\alpha
\sum_{i=1}^{I_{x}}x_{i}^{m-1}\left[ \sum_{m=1}^{N_{x}}\theta _{m}x_{i}^{m-1}%
\right]  \label{17}
\end{gather}%
Then from $\left( \ref{16}\right) ,\left( \ref{17}\right) $\ the expression
of the gradient of the cost function (\ref{4}) is%
\begin{eqnarray}
\nabla _{\phi }S_{\alpha }(\boldsymbol{\phi },\boldsymbol{\theta }) &=&\left[
\frac{\partial S_{\alpha }(\boldsymbol{\phi },\boldsymbol{\theta })}{%
\partial \phi _{1}},...,\frac{\partial S_{\alpha }(\boldsymbol{\phi },%
\boldsymbol{\theta })}{\partial \phi _{N_{t}}}\right]  \label{16_1} \\
\nabla _{\boldsymbol{\theta }}S_{\alpha }(\boldsymbol{\phi },\boldsymbol{%
\theta }) &=&\left[ \frac{\partial S_{\alpha }(\boldsymbol{\phi },%
\boldsymbol{\theta })}{\partial \theta _{1}},...,\frac{\partial S_{\alpha }(%
\boldsymbol{\phi },\boldsymbol{\theta })}{\partial \theta _{N_{x}}}\right] \label{17_1}
\end{eqnarray}

\section{Stability Analysis}

The next lemma gives a necessity condition for stabilization of the
minimizer of the functional $S_{\alpha }(\boldsymbol{\phi },\boldsymbol{%
\theta })$ given by (\ref{4}).

\begin{theorem}
Let $\boldsymbol{\phi }^{\ast },\boldsymbol{\theta }^{\ast }$ be the optimal
solution for $S_{\alpha }(\boldsymbol{\phi },\boldsymbol{\theta })$ and $%
u^{\ast }\left( x,t\right) $ be a solution corresponding to these optimal solution. If for any $\boldsymbol{\phi },\boldsymbol{\theta }$, $v\left(
x,t\right) $ is the solution of the following problem:%
\begin{equation}
\left\{
\begin{array}{l}
\frac{\partial v}{\partial t}=\frac{\partial ^{2}v}{\partial x^{2}}%
+\sum_{k=1}^{N_{t}}\left( \phi _{k}-\phi _{k}^{\ast }\right)
t^{k-1},~~(x,t)\in (0,L)\times (0,t_{f}) \\
v(x,0)=\sum_{m=1}^{N_{x}}\left( \theta _{m}-\theta _{m}^{\ast }\right)
x^{m-1},~~x\in (0,L), \\
v(0,t)=v(L,t)=0,~~t\in (0,t_{f}].%
\end{array}%
\right. ,  \label{stab1}
\end{equation}%
then we have the following estimates:%
\begin{gather}
2\alpha \left( \sum_{i=1}^{I_{x}}\sum_{m=1}^{N_{x}}\theta _{m}\left( \theta _{m}-\theta _{m}^{\ast }\right) x_{i}^{2m-2}+\sum_{j=1}^{I_{t}}\sum_{k=1}^{N_{t}}\phi _{k}\left( \phi
_{k}-\phi _{k}^{\ast }\right) t_{j}^{2k-2}\right) \geq   \notag \\
2\sum_{i=1}^{I_{x}}\left[ u_{f}(x_{i})-u(x_{i},t_{f};\boldsymbol{\phi }%
^{\ast },\boldsymbol{\theta }^{\ast })\right] v\left( x_{i},t_{f}\right)
+\sum_{j=1}^{I_{t}}\left[ u^{\ast }(t_{j})-u(x^{\ast },t_{j};\boldsymbol{%
\phi }^{\ast },\boldsymbol{\theta }^{\ast })\right] v\left( x^{\ast },t_{j}\right) .  \label{stab2}
\end{gather}
\end{theorem}

\begin{proof}
For any $\boldsymbol{\phi },\boldsymbol{\theta }, \gamma\in(0,1)$, let us set the following
parameters:%
\begin{eqnarray*}
\boldsymbol{\phi }^{\gamma } &\boldsymbol{=}&\left( 1-\gamma \right)
\boldsymbol{\phi }^{\ast }+\gamma \boldsymbol{\phi ,} \\
\boldsymbol{\theta }^{\gamma } &=&\left( 1-\gamma \right) \boldsymbol{\theta }^{\ast }+\gamma \boldsymbol{\theta ,}
\end{eqnarray*}%
then%
\begin{eqnarray*}
S_{\alpha }(\boldsymbol{\phi }^{\gamma },\boldsymbol{\theta }^{\gamma })
&=&\sum_{i=1}^{I_{x}}\left[ u_{f}(x_{i})-u(x_{i},t_{f};\boldsymbol{\phi }%
^{\gamma },\boldsymbol{\theta }^{\gamma })\right] ^{2}+\sum_{j=1}^{I_{t}}%
\left[ u^{\ast }(t_{j})-u(x^{\ast },t_{j};\boldsymbol{\phi }^{\gamma },%
\boldsymbol{\theta }^{\gamma })\right] ^{2} \\
&&+\alpha \sum_{i=1}^{I_{x}}\left[ \sum_{m=1}^{N_{x}}\theta _{m}^{\gamma
}x_{i}^{m-1}\right] ^{2}+\alpha \sum_{j=1}^{I_{t}}\left[ \sum_{k=1}^{N_{t}}%
\phi _{k}^{\gamma }t_{j}^{k-1}\right] ^{2}
\end{eqnarray*}%
Since $\boldsymbol{\phi }^{\ast },\boldsymbol{\theta }^{\ast }$ are optimal%
\begin{gather*}
\left. \frac{dS_{\alpha }(\boldsymbol{\phi }^{\gamma },\boldsymbol{\theta }%
^{\gamma })}{d\gamma }\right\vert _{\gamma =0}=-2\sum_{i=1}^{I_{x}}\left[
u_{f}(x_{i})-u(x_{i},t_{f};\boldsymbol{\phi }^{\ast },\boldsymbol{\theta }%
^{\ast })\right] \left. \frac{\partial u(x_{i},t_{f};\boldsymbol{\phi }%
^{\gamma },\boldsymbol{\theta }^{\gamma })}{\partial \gamma
}\right\vert _{\gamma =0} \\
-2\sum_{j=1}^{I_{t}}\left[ u^{\ast }(t_{j})-u(x^{\ast },t_{j};\boldsymbol{%
\phi }^{\ast },\boldsymbol{\theta }^{\ast })\right] \left. \frac{u(x^{\ast
},t_{j};\boldsymbol{\phi }^{\gamma },\boldsymbol{\theta }^{\gamma })}{%
\partial \gamma }\right\vert _{\gamma =0}\\
+2\alpha \sum_{i=1}^{I_{x}}\sum_{m=1}^{N_{x}}\theta _{m}\left( \theta _{m}-\theta _{m}^{\ast }\right) x_{i}^{2m-2}+2\alpha \sum_{j=1}^{I_{t}}\sum_{k=1}^{N_{t}}\phi _{k}\left( \phi _{k}-\phi _{k}^{\ast }\right) t_{j}^{2k-2}\geq 0.
\end{gather*}%
Let us denote $v\left( x,t\right) =\left. \frac{\partial u^{\gamma }(x,t)}{%
\partial \gamma }\right\vert _{\gamma =0}$ which is the solution of  $\left( %
\ref{stab1}\right) $ and using in the above inequality:%
\begin{gather*}
\left. \frac{dS_{\alpha }(\boldsymbol{\phi }^{\gamma },\boldsymbol{\theta }%
^{\gamma })}{d\gamma }\right\vert _{\gamma =0}=-2\sum_{i=1}^{I_{x}}\left[
u_{f}(x_{i})-u(x_{i},t_{f};\boldsymbol{\phi }^{\ast },\boldsymbol{\theta }%
^{\ast })\right] v\left( x_{i},t_{f}\right) +2\alpha \sum_{i=1}^{I_{x}}\sum_{m=1}^{N_{x}}\theta _{m}\left( \theta _{m}-\theta
_{m}^{\ast }\right) x_{i}^{2m-2} \\
-2\sum_{j=1}^{I_{t}}\left[ u^{\ast }(t_{j})-u(x^{\ast },t_{j};\boldsymbol{%
\phi }^{\ast },\boldsymbol{\theta }^{\ast })\right] v\left( x^{\ast },t_{j}\right) +2\alpha \sum_{j=1}^{I_{t}}\sum_{k=1}^{N_{t}}\phi _{k}\left( \phi _{k}-\phi _{k}^{\ast }\right) t_{j}^{2k-2}\geq 0,
\end{gather*}%
which is the required result.
\end{proof}

\section{Conjugate Gradient Method}

Here, we use the Conjugate Gradient Method which is an iterative process to
estimate the parameters $\left\{ \boldsymbol{\phi },\boldsymbol{\theta }%
\right\} $ by minimizing the cost function $S_{\alpha }(\boldsymbol{\phi },%
\boldsymbol{\theta })$ given in $\left( \ref{4}\right) .$ We establish the
new iterate $\left\{ \boldsymbol{\phi }^{\left( n+1\right) },\boldsymbol{%
\theta }^{\left( n+1\right) }\right\} $ from the previous iteration $\left\{ \boldsymbol{\phi }^{\left( n\right) },\boldsymbol{\theta }^{\left( n\right)
}\right\} $ as follows:%
\begin{eqnarray}
\boldsymbol{\phi }^{\left( n+1\right) } &=&\boldsymbol{\phi }^{\left(
n\right) }-\beta _{\phi }d_{\phi }^{\left( n\right) }  \label{18} \\
\boldsymbol{\theta }^{\left( n+1\right) } &=&\boldsymbol{\theta }%
^{\left( n\right) }-\beta _{\theta }d_{\theta }^{\left( n\right) } \label{19}
\end{eqnarray}%
where $\beta _{\phi },\beta _{\theta }$ are the search step size, $n$ shows the iteration number, $d_{\phi }^{\left( n\right) }$ is the direction of
descent for $\boldsymbol{\phi }$\ given as%
\begin{equation}
d_{\phi }^{\left( n\right) }=\left\{
\begin{array}{l}
\nabla _{\phi }S_{\alpha }(\boldsymbol{\phi }^{\left( n\right) },\boldsymbol{%
\theta }^{\left( n\right) }),\text{if }n=0 \\
\nabla _{\phi }S_{\alpha }(\boldsymbol{\phi }^{\left( n\right) },\boldsymbol{%
\theta }^{\left( n\right) })+\gamma _{\phi }^{\left( n\right) }d_{\phi
}^{\left( n-1\right) },\text{if }n>0%
\end{array}%
\right. ,  \label{20}
\end{equation}%
$d_{\theta }^{\left( n\right) }$ is the direction of descent for $%
\boldsymbol{\phi ,}$as well,i.e.%
\begin{equation}
d_{\theta }^{\left( n\right) }=\left\{
\begin{array}{l}
\nabla _{\theta }S_{\alpha }(\boldsymbol{\phi }^{\left( n\right) },%
\boldsymbol{\theta }^{\left( n\right) }),\text{if }n=0 \\
\nabla _{\theta }S_{\alpha }(\boldsymbol{\phi }^{\left( n\right) },%
\boldsymbol{\theta }^{\left( n\right) })+\gamma _{\theta }^{\left( n\right)
}d_{\theta }^{\left( n-1\right) },\text{if }n>0%
\end{array}%
\right. .  \label{21}
\end{equation}%
Here $\gamma _{\phi }^{\left( n\right) }$ and $\gamma _{\theta }^{\left( n\right) }$ are the conjugate coefficients obtained by the direction of descent $d_{\phi }^{\left( n\right) }$ and $d_{\theta }^{\left( n\right) }$
conjugated to the previous one $d_{\phi }^{\left( n-1\right) }$and $%
d_{\theta }^{\left( n-1\right) }$ ,respectively. In the literature, different kind of the conjugate coefficients can be found. The most common
conjugation is the Fletcher--Reeve's (FR) version \cite{Fletcher} given by:%
\begin{eqnarray}
\gamma _{\phi }^{\left( 0\right) } &=&0,~\gamma _{\phi }^{\left( n\right) }=%
\frac{\sum_{k=1}^{N_{t}}\left[ \frac{\partial S_{\alpha }(\boldsymbol{\phi }%
^{\left( n\right) },\boldsymbol{\theta }^{\left( n\right) })}{\partial \phi
_{k}}\right] ^{2}}{\sum_{k=1}^{N_{t}}\left[ \frac{\partial S_{\alpha }(%
\boldsymbol{\phi }^{\left( n-1\right) },\boldsymbol{\theta }^{\left(
n-1\right) })}{\partial \phi _{k}}\right] ^{2}},n>0,  \label{22} \\
\gamma _{\theta }^{\left( 0\right) } &=&0,~\gamma _{\theta }^{\left(
n\right) }=\frac{\sum_{m=1}^{N_{x}}\left[ \frac{\partial S_{\alpha }(%
\boldsymbol{\phi }^{\left( n\right) },\boldsymbol{\theta }^{\left( n\right)
})}{\partial \theta _{m}}\right] ^{2}}{\sum_{m=1}^{N_{x}}\left[ \frac{%
\partial S_{\alpha }(\boldsymbol{\phi }^{\left( n-1\right) },\boldsymbol{%
\theta }^{\left( n-1\right) })}{\partial \theta _{m}}\right] ^{2}},n>0. \label{23}
\end{eqnarray}%
The search step size $\beta _{\phi }$ appearing in equation (\ref{18}) is
obtained by minimizing the function $S_{\alpha }\left( \boldsymbol{\phi }%
^{\left( n+1\right) },\boldsymbol{\theta }^{\left( n\right) }\right) $ with respect to $\boldsymbol{\phi }^{\left( n+1\right) }$, that is, \
\begin{equation}
\beta _{\phi }=\frac{E_{\phi }}{F_{\phi }}  \label{24}
\end{equation}%
where%
\begin{gather}
E_{\phi }=\sum_{i=1}^{I_{x}}\left[ u\left( x_{i},t_{f};\boldsymbol{\phi }%
^{\left( n\right) },\boldsymbol{\theta }^{\left( n\right) }\right)
-u_{f}(x_{i})\right] u\left( x_{i},t_{f};d_{\phi }^{\left( n\right) },%
\boldsymbol{\theta }^{\left( n\right) }\right)   \notag \\
+\sum_{j=1}^{I_{t}}\left[ u\left( x^{\ast },t_{j};\boldsymbol{\phi }^{\left(
n\right) },\boldsymbol{\theta }^{\left( n\right) }\right) -u^{\ast }(t_{j})%
\right] u\left( x^{\ast },t_{j};d_{\phi }^{\left( n\right) },\boldsymbol{%
\theta }^{\left( n\right) }\right)\nonumber \\ +\alpha \sum_{j=1}^{I_{t}}\left( \sum_{k=1}^{N_{t}}\phi _{k}^{\left( n\right) }t_{j}^{k-1}\right) \left( \sum_{k=1}^{N_{t}}\left( d_{\phi }^{\left( n\right) }\right)
_{k}t_{j}^{k-1}\right)   \label{25} \\
F_{\phi }=\sum_{i=1}^{I_{x}}u^{2}\left( x_{i},t_{f};d_{\phi }^{\left( n\right) },\boldsymbol{\theta }^{\left( n\right) }\right)
+\sum_{j=1}^{I_{t}}u^{2}\left( x^{\ast },t_{j};d_{\phi }^{\left( n\right) },%
\boldsymbol{\theta }^{\left( n\right) }\right)\nonumber \\ +\alpha \sum_{j=1}^{I_{t}}%
\left[ \sum_{k=1}^{N_{t}}\left( d_{\phi }^{\left( n\right) }\right) _{k}t_{j}^{k-1}\right] ^{2}  \label{26}
\end{gather}%
and $\beta _{\theta }$ appearing in equation (\ref{19}) is obtained by minimizing the function $S_{\alpha }\left( \boldsymbol{\phi }^{\left( n\right) },\boldsymbol{\theta }^{\left( n+1\right) }\right) $ with respect to $\boldsymbol{\theta }^{\left( n+1\right) }$, that is, \
\begin{equation}
\beta _{\theta }=\frac{E_{\theta }}{F_{\theta }}  \label{27}
\end{equation}%
where%
\begin{gather}
E_{\theta }=\sum_{i=1}^{I_{x}}\left[ u\left( x_{i},t_{f};\boldsymbol{\phi }%
^{\left( n\right) },\boldsymbol{\theta }^{\left( n\right) }\right) -u_{f}(x_{i})\right] u\left( x_{i},t_{f};\boldsymbol{\phi }^{\left( n\right)
},d_{\theta }^{\left( n\right) }\right)   \notag \\
+\sum_{j=1}^{I_{t}}\left[ u\left( x^{\ast },t_{j};\boldsymbol{\phi }^{\left(
n\right) },\boldsymbol{\theta }^{\left( n\right) }\right) -u^{\ast }(t_{j})%
\right] u\left( x^{\ast },t_{j};\boldsymbol{\phi }^{\left( n\right) },d_{\theta }^{\left( n\right) }\right) \notag \\+\alpha \sum_{j=1}^{I_{t}}\left( \sum_{k=1}^{N_{t}}\phi _{k}^{\left( n\right) }t_{j}^{k-1}\right) \left( \sum_{k=1}^{N_{t}}\left( d_{\theta }^{\left( n\right) }\right)
_{k}t_{j}^{k-1}\right)   \label{28} \\
F_{\phi }=\sum_{i=1}^{I_{x}}u^{2}\left( x_{i},t_{f};\boldsymbol{\phi }%
^{\left( n\right) },d_{\theta }^{\left( n\right) }\right) +\sum_{j=1}^{I_{t}}u^{2}\left( x^{\ast },t_{j};\boldsymbol{\phi }^{\left(
n\right) },d_{\theta }^{\left( n\right) }\right)\notag \\ +\alpha \sum_{j=1}^{I_{t}}%
\left[ \sum_{k=1}^{N_{t}}\left( d_{\theta }^{\left( n\right) }\right) _{k}t_{j}^{k-1}\right] ^{2}  \label{29}
\end{gather}

\subsection{Algorithm model for iteratively gradient method:}

The computational procedure for the solution of this inverse problem (\ref%
{1_e})-(\ref{2_2}) using the conjugate gradient method is introduced as follows:

\begin{description}
\item[Step 1] Choose $\boldsymbol{\phi }^{\left( 0\right) },\boldsymbol{%
\theta }^{\left( 0\right) }$, and set $n=0$,

\item[Step 2] Calculate $u\left( x_{i},t_{f};\boldsymbol{\phi }^{\left( n\right) },\boldsymbol{\theta }^{\left( n\right) }\right) $, $i=1,2,...,I_{x}$ in (\ref{10}) and $u\left( x^{\ast },t_{j};\boldsymbol{\phi }^{\left( n\right) },\boldsymbol{\theta }^{\left( n\right) }\right)$, $j=1,2,....,I_{t}$ in (\ref{11}),

\item[Step 3] Compute $\frac{\partial S_{\alpha }(\boldsymbol{\phi }^{\left(
n\right) },\boldsymbol{\theta }^{\left( n\right) })}{\partial \phi _{k}}%
,k=1,2,...,N_{t}$ in (\ref{16}) and $\frac{\partial S_{\alpha }(\boldsymbol{%
\phi }^{\left( n\right) },\boldsymbol{\theta }^{\left( n\right) })}{\partial \theta _{m}},m=1,2,...,Nx$ in (\ref{17}) ,

\item[Step 4] Knowing $\frac{\partial S_{\alpha }(\boldsymbol{\phi }^{\left(
n\right) },\boldsymbol{\theta }^{\left( n\right) })}{\partial \theta _{m}}%
,m=1,2,...,Nx$ and $\frac{\partial S_{\alpha }(\boldsymbol{\phi }^{\left(
n\right) },\boldsymbol{\theta }^{\left( n\right) })}{\partial \phi _{k}}%
,k=1,2,...,N_{t},$ compute the conjugate coefficients $\gamma _{\phi }^{\left( n\right) }$ in (\ref{22}) and $\gamma _{\theta }^{\left( n\right) } $ in (\ref{23}),

\item[Step 5] Compute $d_{\phi }^{\left( n\right) }$ in (\ref{20}) and $%
d_{\theta }^{\left( n\right) }$ in (\ref{21}),

\item[Step 6] By setting $d_{\phi }^{\left( n\right) }$ instead of $%
\boldsymbol{\phi }^{\left( n\right) }$. calculate $u\left( x_{i},t_{f};d_{\phi }^{\left( n\right) },\boldsymbol{\theta }^{\left( n\right) }\right) ,i=1,2,...,I_{x}$ in (\ref{10}) and $u\left( x^{\ast },t_{j};d_{\phi }^{\left( n\right) },\boldsymbol{\theta }^{\left( n\right) }\right) ,j=1,2,....,I_{t}$ in (\ref{11}) and putting $d_{\theta }^{\left(
n\right) }$ instead of $\boldsymbol{\theta }^{\left( n\right) }$ calculate $%
u\left( x_{i},t_{f};\boldsymbol{\phi }^{\left( n\right) },d_{\theta }^{\left( n\right) }\right) ,i=1,2,...,I_{x}$ in (\ref{10}) and $u\left( x^{\ast },t_{j};\boldsymbol{\phi }^{\left( n\right) },d_{\theta }^{\left( n\right) }\right) ,j=1,2,....,I_{t}$ in (\ref{11})

\item[Step 7] Compute the step size $\beta _{\phi }$ form (\ref{24}) by
using (\ref{25}) and (\ref{26}) and the step size $\beta _{\theta }$ form (%
\ref{27}) by using (\ref{28}) and (\ref{29}),

\item[Step 8] Update $\boldsymbol{\phi }^{\left( n+1\right) },\boldsymbol{%
\theta }^{\left( n+1\right) }$ from (\ref{18}) and (\ref{19}),

\item[Step 9] Stop computing if the stopping criterion
\begin{equation*}
S_{\alpha }\left( \boldsymbol{\phi }^{\left( n+1\right) },\boldsymbol{\theta }^{\left( n+1\right) }\right) <\varepsilon
\end{equation*}%
is satisfied. Otherwise set $n=n+1$ and go to Step 2.
\end{description}

\subsection{Results}
Now using the above proposal algorithm we will reconstruct the unknown internal heat source $F(t)$ performed on the mesh $w_{h_t}:=\{t_j: t_0=0, t_j=jh_t, j=1,2,...,I_t, h_t=t_f/I_t\}$ and initial temperature $u_0(x)$ performed on the mesh $w_{h_x}:=\{x_i: x_0=0, x_i=ih_x, j=1,2,...,I_x, h_x=L/I_x\}$. We use the stopping criteria as $\varepsilon:=10^{-3}$. The root mean square errors are computed by $E_F=\sqrt{\frac{\sum_{j=0}^{I_t} \left(F(t_j)-F^{h_t}(t_j) \right)^2}{I_t}}$ and $E_{u_0}=\sqrt{\frac{\sum_{i=0}^{I_x} \left(u_0(x_i)-u_0^{h_x}(x_i) \right)^2}{I_x}}$ where $F(t_j)$ and $u_0(x_i)$ are exact solution of the inverse problem and $F^{h_t}(t_j)$ and $u_0^{h_x}(x_i)$ are numerical solution of the inverse problem.

\begin{example}\label{ex1} In the first numerical experiment we take
the exact solution $u(x,t)=(\sin(x)+1)\exp(-t),~(x,t)\in\left(-\frac{\pi}{2},\frac{3\pi}{2}\right)\times(0,2),$ with the internal heat source $F(t)=-\exp(-t)$ and initial temperature $u_0(x)=sin(x)+1$. Figure \ref{ex1:fig1}  shows the solutions obtained with the iterative algorithm. Table \ref{ex1:tab1}  gives the comparison of the results with respect to $N_t$, $N_x$ and the interior point $x^*$. It is also observed from Table \ref{ex1:tab1} that for $N_t=9$ and $N_x=12$ the lowest errors values of $E_F$ and $E_{u_0}$ are determined in case of $x^*=2.97$ which is close to the right boundary.
\begin{figure}[H]
  \begin{subfigure}[b]{0.5\textwidth}
          \centering
          \includegraphics[width=\textwidth]{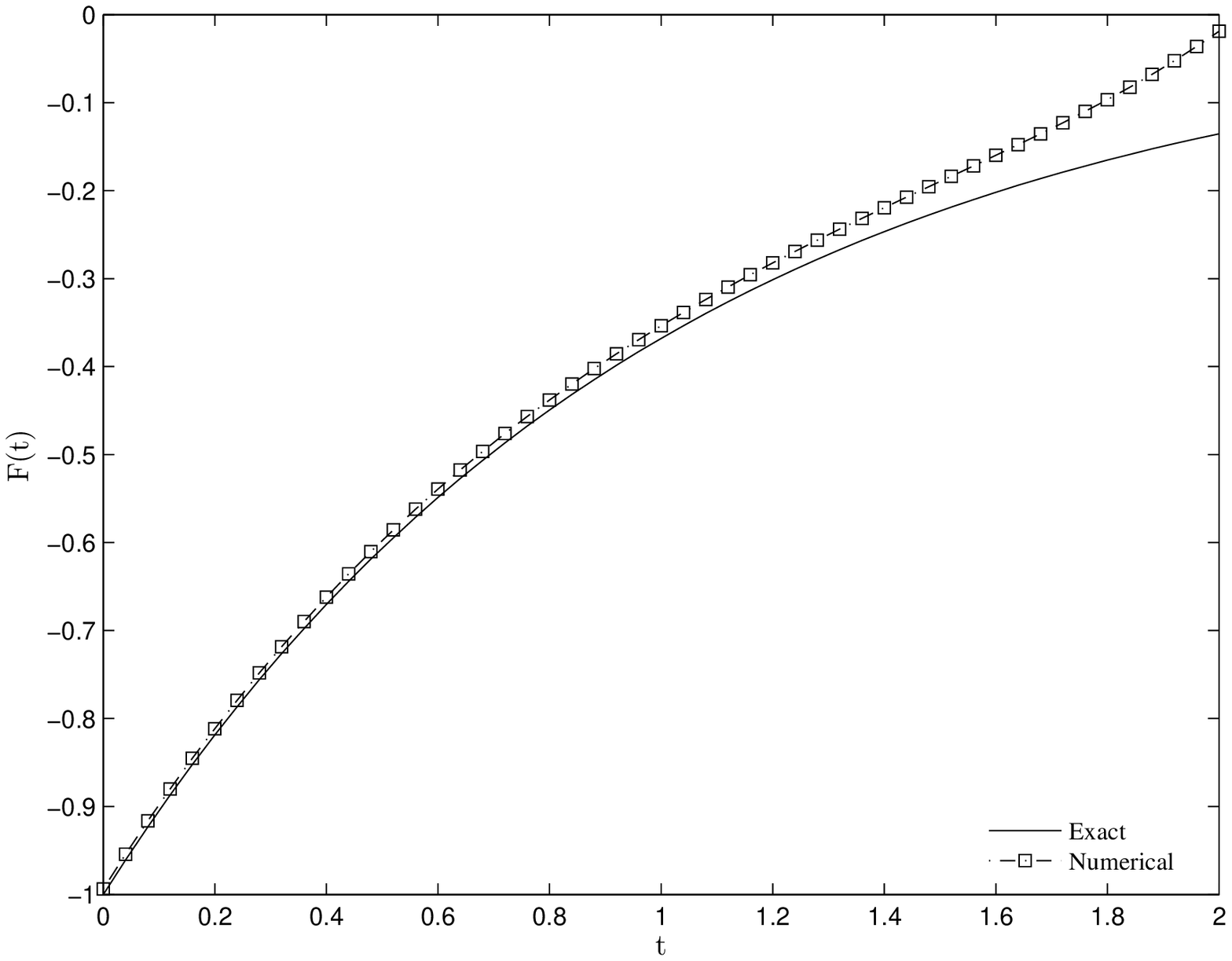}
          \caption{}{\label{ex1:fig1-a}}
  \end{subfigure}%
  \begin{subfigure}[b]{0.53\textwidth}
          \centering
          \includegraphics[width=\textwidth]{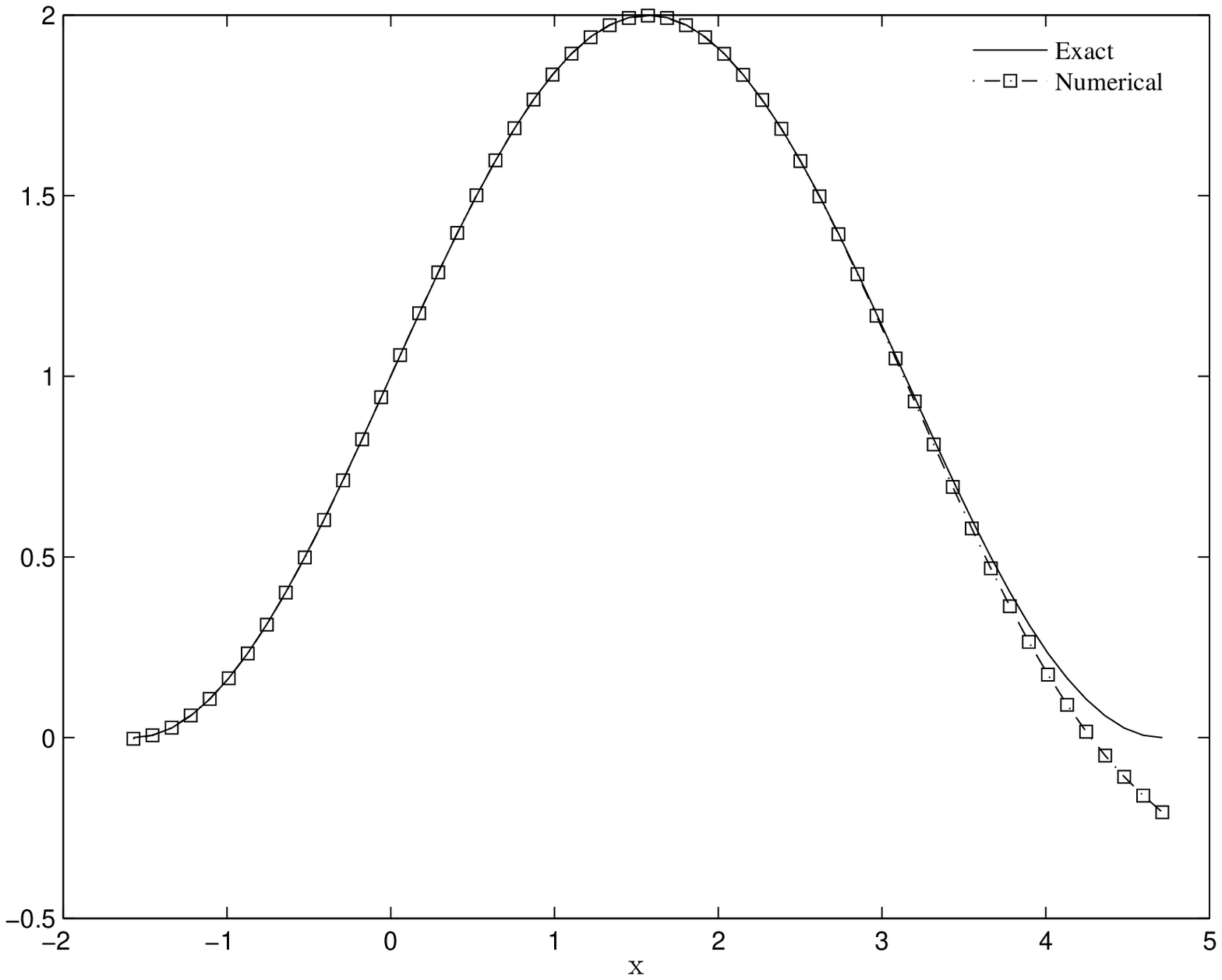}
          \caption{}{\label{ex1:fig1-b}}
  \end{subfigure}
  \caption{(\subref{ex1:fig1-a}) shows the analytical and numerical solutions
  of $F(t)$; (\subref{ex1:fig1-b}) presents the analytical and numerical solutions of $u_0(x)$
  when $N_t=9$, $N_x=12$ and $x^*=2.97$.}\label{ex1:fig1}
\end{figure}
\begin{table}[H]
\centering
\begin{tabular}{|c|c|c|c|}
  $N_x\times N_t$ & $x^*$ & $E_F$ & $E_{u_0}$
\\\hline \multirow{3}{*}{$6 \times 5$} & $-1.34$ & $8.09\times 10^{-3}$ & $7.14\times 10^{-2}$  \\
 \cline{2-2}\cline{3-2}\cline{4-2} & $-0.17$ & $7.62\times 10^{-3}$ & $7.12\times 10^{-2}$  \\
 \cline{2-2}\cline{3-2}\cline{4-2} & $0.99$ & $7.64\times 10^{-3}$ & $6.69\times 10^{-2}$ \\
 \cline{2-2}\cline{3-2}\cline{4-2} & $2.15$ & $7.59\times 10^{-3}$ & $5.82\times 10^{-2}$ \\
 \cline{2-2}\cline{3-2}\cline{4-2} & $2.97$ & $7.37\times 10^{-3}$ & $4.93\times 10^{-2}$ \\
 \hline \multirow{3}{*}{$12 \times 9$} & $-1.34$ & $7.44\times 10^{-3}$ & $5.23\times 10^{-2}$  \\
 \cline{2-2}\cline{3-2}\cline{4-2} & $-0.17$ & $6.53\times 10^{-3}$ & $4.82\times 10^{-2}$  \\
 \cline{2-2}\cline{3-2}\cline{4-2} & $0.99$ & $6.54\times 10^{-3}$ & $4.55\times 10^{-2}$ \\
 \cline{2-2}\cline{3-2}\cline{4-2} & $2.15$ & $5.25\times 10^{-3}$ & $3.78\times 10^{-2}$ \\
 \cline{2-2}\cline{3-2}\cline{4-2} & $2.97$ & $5.03\times 10^{-3}$ & $2.98\times 10^{-2}$ \\
 \hline
\end{tabular}
\caption{Errors with various values of $N_t$, $N_x$ and $x^*$ for Example \ref{ex1}.}\label{ex1:tab1}
\end{table}

\end{example}

\newpage
\bibliographystyle{model1b-num-names}
\bibliography{manuscript_ref}

\end{document}